\documentclass[12pt]{article}

\usepackage{mathrsfs}
\usepackage{amsfonts}
\usepackage[leqno]{amsmath}
\usepackage{graphicx}
\usepackage{latexsym}
\usepackage{amsmath,amsfonts,amssymb,amsthm,mathrsfs,euscript,%makeidx,
color}
\usepackage{enumerate}

\def\sqr#1#2{{\vcenter{\vbox{\hrule height.#2pt
				\hbox{\vrule width.#2pt height#1pt \kern#1pt \vrule width.#2pt}
				\hrule height.#2pt}}}}
\def\5n{\negthinspace \negthinspace \negthinspace \negthinspace \negthinspace }
\def\4n{\negthinspace \negthinspace \negthinspace \negthinspace }
\def\3n{\negthinspace \negthinspace \negthinspace }
\def\2n{\negthinspace \negthinspace }
\def\1n{\negthinspace }

\def\dbE{\mathbb{E}}

\def\dbN{\mathbb{N}}

\def\dbR{\mathbb{R}}

\def\dbT{\mathbb{T}}

\def\sM{\mathscr{M}}

\def\cF{{\cal F}}

\def\cM{{\cal M}}

\def\cP{{\cal P}}
\def\cQ{{\cal Q}}

\def\cU{{\cal U}}

\def\BA{{\bf A}}

\def\BI{{\bf I}}
\def\BJ{{\bf J}}

\def\BP{{\bf P}}

\def\BV{{\bf V}}

\def\ds{\displaystyle}

\def\ns{\noalign{\ss}}

\def\ms{\medskip}

\def\q{\quad}
\def\qq{\qquad}

\def\({\Big (}
\def\){\Big )}
\def\[{\Big[}
\def\]{\Big]}

\def\d{\delta}
\def\e{\varepsilon}

\def\k{\kappa}

\def\t{\tau}

\def\D{\Delta}

\def\bde{\begin{definition}\label}
	\def\ede{\end{definition}}
\def\be{\begin{equation}}
\def\bel{\begin{equation}\label}
\def\ee{\end{equation}}
\def\bt{\begin{theorem}\label}
	\def\et{\end{theorem}}
\def\bc{\begin{corollary}\label}
	\def\ec{\end{corollary}}
\def\bl{\begin{lemma}\label}
	\def\el{\end{lemma}}
\def\bp{\begin{proposition}\label}
	\def\ep{\end{proposition}}
\def\bas{\begin{assumption}\label}
	\def\eas{\end{assumption}}
\def\br{\begin{remark}\label}
	\def\er{\end{remark}}
\def\bex{\begin{example}\label}
	\def\ex{\end{example}}
\def\ba{\begin{array}}
	\def\ea{\end{array}}
\def\ben{\begin{enumerate}}
	\def\een{\end{enumerate}}
\def\square#1{\vbox{\hrule\hbox{\vrule height#1%
			\kern#1\vrule}\hrule}}
\def\rectangle#1#2{\vbox{\hrule\hbox{\vrule height#1%
			\kern#2\vrule}\hrule}}

\font\tenbb=msbm10 \font\sevenbb=msbm7 \font\fivebb=msbm5

\newfam\bbfam
\scriptscriptfont\bbfam=\fivebb \textfont\bbfam=\tenbb
\scriptfont\bbfam=\sevenbb

\newtheorem{theorem}{\indent Theorem}[section]
\newtheorem{definition}[theorem]{\indent Definition}
\newtheorem{proposition}[theorem]{\indent Proposition}
\newtheorem{corollary}[theorem]{\indent Corollary}
\newtheorem{lemma}[theorem]{\indent Lemma}
\newtheorem{remark}[theorem]{\indent Remark}
\newtheorem{example}[theorem]{\indent Example}

\newtheorem{assumption}[theorem]{\indent Assumption}
\makeatletter

\@addtoreset{equation}{section}
\makeatother

\def\bea{\begin{equation*}}
\def\eea{\end{equation*}}

\def\bel{\begin{equation}\label}
\def\eel{\end{equation}}

\def\ba{\begin{array}}
	\def\ea{\end{array}}
\newcommand{\ad}{&\!\!\!\displaystyle}

\def\({\Big (}
\def\){\Big )}
\def\[{\Big[}
\def\]{\Big]}
\def\q{\quad}
\def\qq{\qquad}

\def\d{\delta}
\def\e{\varepsilon}

\def\ms{\vspace{0.2cm}}
\def\ds{\displaystyle}

\def\ns{\noalign{\smallskip}}

\def\argmin{\mathop{\rm argmin}}

\begin{document}
\title{Time-Inconsistent Problems for Controlled Markov Chains with Distribution-Dependent Costs: Equilibrium Solutions}
%{Time-inconsistent  Mean-field Equilibrium for Continuous-time Markov Chains}
\author{Hongwei Mei\thanks{Department of Mathematics, The University of Kansas, Lawrence, KS 66045, U.S. (hongwei.mei@ku.edu).}~~~ and ~~George Yin \thanks{Department of  Mathematics, Wayne State University, Detroit, MI 48202, gyin@wayne.edu. The research of this author was supported in part by the
		Air Force Office of Scientific Research under grant FA9550-18-1-0268.}}
\maketitle

\begin{abstract}
This paper focuses on a class of continuous-time  controlled Markov chains with time-inconsistent and distribution-dependent cost functional (in some appropriate sense). A new definition of time-inconsistent distribution-dependent  equilibrium in closed-loop sense is given and its existence and uniqueness have been established. Because of the time-inconsistency, it is  proved that the equilibrium is locally optimal in an appropriate sense. Moreover, it has been shown that our problem is essentially equivalent to an infinite-player  mean-field game with time-inconsistent cost.
\end{abstract}

\paragraph{Keywords} time-inconsistent, distribution-dependent, mean-field, Markov chain.
\newpage

\section{Introduction}
Because of 
the wide range of applications in existing and emerging 
areas such as
control engineering, biology and ecology, communication systems,
social networks, and
finance and economics, stochastic control problems have been studied extensively 
%by lots of mathematicians 
since last century.
Recently, 
time-inconsistent control problems  have attracted increasing attention.
%arouses lots of  interests  of mathematicians. 
%Different from  
In contrast to 
a time-consistent problem, for a time-inconsistent problem, it is  impossible  to find an optimal control or strategy based on the dynamic programming principle.
%now which stays optimal in the future.
%  This is the so-called time-inconsistency. 
In 
 economics and financial models, there are 
 %multiple
  many reasons for time-inconsistency. For example,  if the cost functional or reward  is non-exponential discounting (see \cite{Yong2012b}) or it is of mean-variance form, the problem is time-inconsistent. 
  %For such important feature, 
  It can be seen that 
  time-inconsistency appears in a wide variety of applications; see \cite{Cheb2020,Dodd2008,Green2003,Gren2007} and references therein.  One of the ideas to deal with time-inconsistency is to find a  strategy,  which is locally optimal only. 
  %instead of a global one to save for the future. 
  After the breakthrough using equilibria was introduced in \cite{Ekeland2008,Ekeland2010,Ekeland2012,Yong2012b},  many works are published concerning time-inconsistent problems; see the survey paper \cite{Yan2019} and the references therein.  Generally speaking, the strategies can be divided into two categories, open-loop  \cite{Ekeland2010} and  closed-loop \cite{Yong2012b}. In the open-loop setup, the main effort is devoted to deriving a local maximal principle (or variational principle) by  the rule of local optimality. In contrast, in the closed-loop one, through the investigation on an $N$-player game, one can derive a  time-inconsistent Hamilton-Jacobi equation to determine a time-inconsistent strategy which verifies a local optimality \cite{Yong2012b}.  For similar closed-loop problems, time-inconsistent control problems with recursive cost functions are considered in \cite{Wei2017} and  time-inconsistent control problems with switching states are considered in \cite{WeiJ2017,Mei2019}. In the closed-loop problems, one of the difficulties lies in the first-order regularity of the viscosity solution to the Hamilton-Jacobi (HJ for short) equation (or it is required that the HJ equation exists a unique classical solution). Therefore, the stochastic differential equations are assumed to be non-degenerate in the most of previous works.  

In this paper, we consider a  controlled continuous-time Markov chain with time-inconsistent and distribution-dependent cost functional in a closed-loop setup. That is, not only does the cost functional depend on a non-exponential discounting factor, but also depends on the distribution of the Markov chain. Compared to the time-inconsistent distribution-independent problem, the distribution-dependence of the cost functional  brings some new features 
%of interests 
and difficulties.
 %to deal with. 
 For example,  the classical dynamic programming is not applicable
 %working anymore 
 and the  strategy is not state-dependent only when finding a optimal control. To overcome the difficulties due to the distribution-dependence, the first idea  is to lift the dynamic into the space of probability measures on which one can derive an HJ equation using the Bellman's principle. Then using mass-transport theory, one can prove the existence and uniqueness of the viscosity solutions (e.g., see \cite{Pham2017,Pham2018} for time-consistent problems). While for time-inconsistent problems, it is required that  there exists a unique classical solution of
 the  HJ equation on the space of probability measures to identify the time-inconsistent strategy.  This is an even harder problem to deal with. Thus we present a different approach in this paper.

 An alternative idea to deal with distribution-dependence can be paralleled from the recent developments in the mean-field game theory (see \cite{Larsy2006a,Larsy2006b,Larsy2007a,Larsy2007b,
Huang2006,Huang2007,Huang2007a,Huang2007b}).
   In a mean-field game problem,   
   %the authors 
   one considers a backward Hamilton-Jacobi equation coupled with a forward  transport equation on the space of probability measures. Using the fixed-point theory,  one can   derive a mean-field equilibrium, which is  essentially a Nash equilibrium point.  Applying the similar idea to the  control problem  with time-consistent and distribution-dependent cost functional,
the first step is to solve a classical HJ equation  which is concluded from the classical optimal control  of system  given a guiding (fixed) process $\rho(\cdot)$ in the cost functional. A  feedback control can be determined if the HJ equation is ``regular enough''. Then the second step is to verify that  the guiding process $\rho(\cdot)$ coincides with the distribution law of the dynamics  using the feedback control. If the two-step verification is fulfilled, the feedback control is called a mean-field equilibrium. One can see that the mean-field equilibrium is essentially a fixed point using such definition. However we have to mention that  the equilibrium  is not an optimal strategy in general even for the time-consistent cases.

In this paper, we 
%deal with time-inconsistency and distribution-dependence   for
considered continuous-time and finite-state controlled   Markov chains  with time-inconsistent distribution-dependent cost functional using a similar idea to mean-field game theory.  We present a new definition for time-inconsistent mean-field equilibrium,  which verifies a local optimality in some appropriate sense.  Then we can prove the existence and uniqueness of the time-inconsistent mean-field equilibrium under some appropriate conditions.  Moreover,  we show  that the time-inconsistent mean-field equilibrium found is equivalent to an equilibrium for a mean-field game of infinite-many equivalent players with time-inconsistent cost functional.   Previous works on time-inconsistent distribution-dependent systems include \cite{Ni2018a,Wang2018,Yong2015}.
 %for example. 
 To the best of our knowledge, most of the papers  are concerned with a special class of linear stochastic differential equations with quadratic costs. There are few papers concerned with  the theory on continuous-time  controlled Markov chains with time-inconsistent and distribution-dependent cost. The paper aims to fill this gap.

 The paper is arranged as follows. In Section \ref{sec:pre}, we introduce some notation together with certain preliminary results for controlled Markov Chain. Section \ref{sec:tcp} introduces the time-inconsistent mean-field equilibrium and prove the existence and uniqueness.
 Section \ref{sec:mfg} shows that our problem is equivalent to an infinite-player mean-field game with time-inconsistent cost. Finally some concluding remarks are made in Section \ref{sec:cr}.

\section{Preliminaries}\label{sec:pre}
\subsection{Notation}
 Let $M=\{1,\cdots,m\}$ and $\dbT=[0,T]$. Denote by $\cM$ be the set of all possible functions 
 defined on $M$ equipped with the sup-norm $\Vert\cdot\Vert_\cM$.  Set
$$C([0,T],\cM):=\left\{\theta:\dbT\times M\mapsto \dbR\big|~\theta_t(i)\text{ is continuous w.r.t. } t\right\}$$
and
$$D([0,T],\cM):=\left\{\theta:\dbT\times M\mapsto \dbR\big|~\theta_t(i)\text{ is right-continuous with left-limit w.r.t. } t\right\}.$$
Let $\cP$ be the collection of all probability measures on $M$ equipped with metric $d(\cdot,\cdot)$ defined by
$$d(\rho,\gamma):=\sum_{i=1}^m|\rho(i)-\gamma(i)|.$$
Denote by
 $C([0,T],\cP)$ be the set of all continuous $\cP$-valued curves   on $\dbT$.
Let $U$ be the space of actions equipped with the metric $|\cdot,\cdot|_U$ and $v_0$ be a fixed element in $U$.  Let $\cU$ be the set of all possible maps from $M$ to $U$ equipped with the following metric
$$d_\cU( u,u')=\sup_{i\in M}|u(i),u'(i)|_U.$$

\subsection{Controlled Markov Chain}
In this paper, we consider a finite-state controlled   Markov chain with generator $Q_t^v=[q_t^{v}(i,j)]_{M\times M}$ satisfying
\bel{dynamic}\frac{d\mu_t}{dt}=\mu_tQ_t^v,\q  t\in\dbT, v\in U,\eel
where the state space is $M$ and the action  space is $U$.

 To make sure $Q_t^v$  is a generator of a Markov chain, we define the admissible action set for state $i$ being
$$U_t(i):=\{v\in U:\text{ $q^v_t(i,\cdot)$ is a generator} \},$$
where $q^v_t(i,\cdot)$ is a generator  in the sense that $q_t^v(i,j)\geq 0$ if $j\neq i$ and $\sum_{j=1}^mq_t^v(i,j)=0$. Throughout the paper, we assume that $U_t(i)$ is not empty and measurable  under the topology of $U$ induced by the metric $|\cdot,\cdot|_U$.

Let the set of all admissible strategies on time-interval $[t_0,T]$ be
$$L^1([t_0,T],\cU):=\{\pi\big|\pi:(s,i)\in[t_0,T]\times M\mapsto\pi_s(i)\in U_t(i)\text{ with }\int_{t_0}^T|\pi_s(i),v_0|_U ds<\infty\}.$$
%Define a subset $\cU^\dagger[t_0,T]$ of $ L^1([t_0,T],\cU) $ by
%$$\cU^\dagger[t_0,T]=\{\pi\in  L^1([t_0,T],\cU) \big|~\pi_s(i)\text{ is right-continuous with respect to $t$ with left-limit for each $i\in M$}.\}$$
%
Define a subset $D([t_0,T],\cU)\subset L^1([t_0,T],\cU)$ by
$$\ba{ll}\ad D([t_0,T],\cU)\\
\ns\ad\qq:=\{\pi\in L^1([t_0,T],\cU) :\pi_{t}(i)\text{ is right-continuous with left-limit w.r.t. $t$ for each }i\in M   \}.\ea$$
We can  define a map $\phi_t: L^1([0,T],\cU) \mapsto L^1([t,T],\cU) $ by
$$(\phi_t[\pi])_s(i):=\pi_s(i),\q\text{for } i\in M,~t\leq s\leq T.$$

Given any $\pi\in  L^1([0,T],\cU) $, we write $\mu_t^{t_0,\rho,\pi}$ as the solution of
\eqref{dynamic} at time $t$ with initial data $\rho$ at time $t_0$ under the strategy $\pi$.  If the initial data $\rho=\d_x$, i.e., 
the Dirac measure concentrated
%full mass
on the point $x$, we write  $\mu_t^{t_0,x,\pi}=\mu_t^{t_0,\d_x,\pi}$. Especially the initial time $t_0$ will be omitted if $t_0=0$.
Now we pose
%raised 
the following assumptions to guarantee the regularity of the dynamic \eqref{dynamic}.

\noindent{\bf Assumption (A)} 

\begin{itemize}
\item[(A1)] The admissible action set $ U_{\cdot}$ is right continuous with left limits. That is, $$\ba{ll}\ds U_t(i)=\lim_{\e\rightarrow 0^+}U_{t+\e}(i)
\text{ and }U_{t-}(i):=\lim_{\e\rightarrow 0^-}U_{t+\e}(i) \text{ exists}.\ea$$

\item[(A2)] $q_{\cdot}^v(i,j)$ is right continuous with left limit in that for any $v_{n}\in U_{t+\e_n}(i)$ with limit $v\in U_{t}(i)$ as $\e_n\rightarrow 0^+$, $$\lim_{n\rightarrow\infty}q_{t+\e_n}^{v_{n}}(i,j)=q_t^v(i,j).$$
For any $v_{n}\in U_{t+\e_n}(i)$ with limit $v\in U_{t-}(i)$ as $\e_n\rightarrow 0^-$, $$q_{t-}^v(i,j):=\lim_{n\rightarrow\infty}q_{t+\e_n}^{v_{n}}(i,j)\text{ exists}.$$

\item[(A3)] There exists  constants $K_1,\kappa_1>0$ such that for any $t\in[0,T]$ and $i,j\in M$,
\bel{qLip}\left\{\ba{ll}\ad|q_t^v(i,j)-q_t^{v'}(i,j)|\leq \kappa_1 |v-v'|_U,\q \text{for }v,v'\in U_t(i);\\ [2mm]
\ns\ad \sup_{j\in M}\sup_{v\in U_t(j)}|q_t^v(i,j)|\leq K_1.\ea\right. \eel
\end{itemize}

\begin{remark}{\rm  Note that $q_t^v(i,j)$ is right continuous with left limit w.r.t. $t$ from (A1) and (A2) and $q_t^v(i,j)$ is Lipschitz with respect to $v$ from (A3).  Thus the existence and uniqueness of the solution of \eqref{dynamic}
given any $\pi\in L^1([0,T],\cU) $ holds directly. 
%Here we have to mention 
Note 
that such assumptions 
are stronger than necessary for 
%too strong for 
the existence and uniqueness of  the solution. The reason we assume such assumption is to conclude the following proposition 
%which will 
to be used 
%when considering 
for the control problem.
}\end{remark}

\begin{proposition} Under Assumption {\rm (A)},
	for any $\pi\in  D([0,T],\cM)$,
	$q_{\cdot}^{\pi_{\cdot}(i)}(i,j)$ is right continuous with left limit for any fixed $i,j\in M$.

\end{proposition}

 To proceed, we present some regularity results about \eqref{dynamic} first. The proof of which  can be found in \cite{Yin2012} (Page 19, Theorem 2.5).

\begin{proposition} Under Assumption {\rm (A)}, for any $\pi\in L^1([0,T],\cU) $, let \bel{BQtransition}\BP_{t_0,t_1}^\pi=\exp\left\{\int_{t_0}^{t_1}Q_s^{\pi_s}ds\right\}:=(p_{t_0,t_1}^{\pi}(i,j))_{M\times M}.\eel
	The unique solution of \eqref{dynamic} is $\mu^{\rho,\pi}_t=\rho\BP^\pi_{0,t}$,
	i.e., for each $j\in M$,  $$\mu^{\rho,\pi}_t(j)=\sum_{j=1}^m p_{0,t}^\pi(i,j)\rho(i).$$

	If $\pi\in L^1([0,T],\cU) $  for any $h\in\cM$, then
%it follows that
\bel{generator}\lim_{\e\rightarrow 0^+}\frac1\e\(\sum_{j=1}^mh(j)\mu_{t+\e}^{t,i,\pi}(j)- h(i)\)=\sum_{j=1}^mh(j)q^{\pi(i)}_{t}(i,j)\text{ holds for a.e. }\ t\in\dbT.\eel
 Moreover, if  $\pi\in D([0,T],\cU)$, \eqref{generator} holds for all $t\in\dbT$.
\end{proposition}

Because of the Lipschitz-dependence of the transition 
rate
 (which depends on the strategies), the following lemma indicates  that the solution 
 of dynamic \eqref{dynamic}
 is Lipschitz-dependent on the strategies as well.

\begin{lemma} Under Assumption {\rm (A)},  given any two admissible strategies $\pi,\pi'\in  L^1([0,T],\cU) $ and $\rho,\gamma \in\cP$, it follows that
	\bel{estimu} d(\mu^{\rho,\pi}_t,\mu^{\gamma,\pi'}_t)\leq d(\rho,\gamma)+ \kappa_1 \int_0^td_\cU(\pi_s,\pi'_s) ds.\eel
	
\end{lemma}

\begin{proof}
Noting that $\sum_{j=1}^m p_t^\pi(i,j)=1$,
 simple calculation yields
\bel{qcontraction}\ba{ll}d(\mu^{\rho,\pi}_t,\mu^{\gamma,\pi}_t)\ad= d(\rho\BP_t^\pi,\gamma\BP_t^\pi)\\
\ns\ad\leq\sum_{i=1}^M\sum_{j=1}^M|\rho(i)-\gamma(i)|p^\pi_{t}(i,j)\\
\ns\ad\leq d(\rho,\gamma).\ea\eel
By \eqref{qLip}, for any $\varepsilon>0$, there exists a $\d_\varepsilon>0$ small enough (independent of $t$) such that  when $0\leq t\leq \d_\varepsilon$,
\bel{qcontraction-2}\ba{ll}d(\mu^{\rho,\pi}_t,\mu^{\rho,\pi'}_t)\ad= d(\rho\BP_t^{\pi},\rho\BP_t^{\pi'})\\
\ns\ad \leq \sum_{j=1}^M\left|\sum_{i=1}^M\rho(i)[p_t^\pi(i,j)-p_t^{\pi'}(i,j)]\right|\\
\ns\ad\leq (1+\e )\kappa_1d_\cU(\pi_s,\pi'_s) ds. \ea\eel
By \eqref{qcontraction},
\bel{qcontraction-3}\ba{ll}d(\mu^{\rho,\pi}_t,\mu^{\gamma,\pi'}_t)\ad \leq d(\mu^{\rho,\pi}_t,\mu^{\rho,\pi'}_t)+d(\mu^{\rho,\pi'}_t,\mu^{\gamma,\pi'}_t)\\
\ns\ad\leq (1+\e )\k_1 \int_0^td_\cU(\pi_s,\pi'_s) ds+d(\rho,\gamma).\ea\eel

For any $t\in [\d_\varepsilon,T]$, using \eqref{qcontraction-2} and \eqref{qcontraction-3},  we have
%it follows that
$$\ba{ll}d(\mu^{\rho,\pi}_t,\mu^{\gamma,\pi'}_t)\ad\leq d(\mu^{\rho,\pi}_{t-\d_\varepsilon},\mu^{\gamma,\pi'}_{t-\d_\e })+(1+\e )\k_1\int_{t-\d_\varepsilon}^td_\cU(\pi_s,\pi'_s) ds \ea $$
By simple recursions and the arbitrariness of $\varepsilon>0$, one can easily see that
$$d(\mu^{\rho,\pi}_t,\mu^{\gamma,\pi'}_t)\leq d(\rho,\gamma)+ \kappa_1 \int_0^td_\cU(\pi_s,\pi'_s)ds.
$$
The proof is complete.

\end{proof}

\section{Time-inconsistent Distribution-dependent Control Problem}\label{sec:tcp}
In this section, we  introduce  the control problem
 that we are interested in and give the  definition of time-inconsistent  equilibrium. Moreover, we  prove the existence and uniqueness of the  time-inconsistent mean-field equilibrium under some general conditions. This section is divided into several subsections.
 
\subsection{Definition of Time-inconsistent Mean-field Equilibrium}
In this subsection, we mainly introduce the definition of  time-inconsistent mean-field equilibrium for our control problem. We begin by defining the time-inconsistent distribution dependent cost functional first.

 Let the running cost and terminal cost rates be maps defined as
$$\left\{\ba{ll}\ad f:\dbT\times \dbT\times M\times U\times\cP\mapsto\dbR^+ \text{ by }(\t,t,i,v,\rho)\mapsto f_{\t,t}(i,v;\rho), \\
\ns\ad g:\dbT\times M\times\cP\mapsto \dbR^+\text{ by } (\t,i,\rho)\mapsto g_{\t}(i;\rho).\ea\right.$$
In the paper, we are concerned with the following  distribution-dependent time-inconsistent cost functional with non-exponential discounting factor $\t$,
\bel{cost}\BJ_{\t,t}(\rho,\pi):=\int_t^T
\sum_{j=1}^mf_{\t,s}(j,\pi_s(j);\mu_s^{t,\rho,\pi})
\mu_s^{t,\rho,\pi}(j)ds+\sum_{j=1}^mg_\t(j;\mu_T^{t,\rho,\pi})\mu_T^{t,\rho,\pi}(j)\eel
and the corresponding  value functional
\bel{value}\BV_t(\rho,\pi):=\BJ_{t,t}(\rho,\pi).\eel

Here the $\BJ$ and $\BV$ are distribution-dependent because $f$ and $g$ depends on the distribution term as well. If $f$ and $g$ are distribution-independent, such problem reduces to a classical time-inconsistent control problem. 
%Then we would call such case being distribution-independent.

As alluded to 
%introduced 
in the introduction, it is impossible to find a optimal strategy because of the time-inconsistency. Thus  we look for the {\it time-inconsistent distribution-dependent equilibrium}.
The definition is given as follows.

\begin{definition}\label{timfe}
	{\rm Given {\it a priori} $\nu_\dbT\in C([0,T],\cP)$,  write
		$$J_{\t,t}(i,\phi_t[\pi];\nu_\dbT):=
		\int_t^T\sum_{j=1}^mf_{\t,s}(j,u_s(j);\nu_s)
\mu_s^{t,i,\pi}(j)ds+\sum_{j=1}^mg_\t(j;\nu_T)\mu_T^{t,i,\pi}(j)$$
		and		$$V_t(i,\phi_t[\pi];\nu_\dbT)=J_{t,t}(i,\phi_t[\pi];\nu_\dbT).$$
		A pair $(\rho,\pi)\in\cP\times D([0,T],\cU)$  is called a {\it time-inconsistent equilibrium} for our distribution-dependent cost if the following local-optimality holds,
		\bel{localoptimal}\ba{ll}\ad\limsup_{\e\rightarrow 0^+}\frac{V_t(i,\pi^\e\oplus\phi_{t+\e}[\pi];\mu_\dbT^{\rho,\pi})-V_t(i,\phi_t[\pi];\mu_\dbT^{\rho,\pi})}\e\geq0, \\
		\ns\ad\qq\qq\qq\qq\qq\qq\qq\text{ for any }(t,i,\pi^\e)\in\dbT\times M\times D([t,t+\e),\cU),\ea\eel
		where the perturbed strategy $\pi^\e\oplus\phi_{t+\e}[\pi]\in D([t,T],\cU) $ is defined as
		$$\(\pi^\e\oplus\phi_{t+\e}[\pi]\)_s(i):=\left\{\ba{ll}\ad\pi_s(i),\q t+\e\leq s\leq T,\\ [1mm]
		\ns\ad \pi^\e_s(i),\q t\leq s< t+\e.\ea\right.$$

}
\end{definition}

\begin{remark}\label{rem2step}
{\rm  We note the following facts.
\begin{itemize}
\item[(1)] Along the solution curve $\mu_\dbT^{\rho,\pi}$, we have $$\BJ_{\t,t}(\mu^{\rho,\pi}_t,\pi)=\sum_{j=1}^m\mu^{\rho,\pi}_t(j)
J_{\t,t}(j,\phi_t[\pi];\mu^{\rho,\pi}_\dbT)\text{ and }\BV_t(\mu^{\rho,\pi}_t,\pi)=\sum_{j=1}^m\mu^{\rho,\pi}_t(j)V_{t}(j,\phi_t[\pi];\mu^{\rho,\pi}_\dbT).$$

\item[(2)] From the definition of a time-inconsistent equilibrium $(\rho,\pi)$, we can see that there are two steps to verify.

\begin{itemize}
\item[(a)] Find the unique solution $\mu_\dbT^{\rho,\pi}$ of  the dynamic \eqref{dynamic} using $(\rho,\pi)$.

\item[(b)] Using $\nu_\dbT=\mu_\dbT^{\rho,\pi}$  as a priori curve, find the distribution-independent time-inconsistent equilibrium strategy which verifies the local-optimality.

\end{itemize}
If such two-step recursion is fulfilled, one can see that $(\rho,\pi)$  is a time-inconsistent mean-field equilibrium. Thus we will adopt the  fixed-point  theory to prove the existence and uniqueness of the time-inconsistent distribution-dependent equilibrium.
\end{itemize}
}
\end{remark}

To proceed, we first
%Here we have to 
present the following example to avoid
%illustrate 
some possible confusions.

\begin{example}
	{\rm Suppose that the terminal cost in \eqref{cost} is given by
	$$g_\t(i,\rho)=(i-\sum_{j=1}^mj\rho(j))^2.$$
 Then
 $\sum_{i=1}^mg_\t(i,\mu_T)\mu_T(i)$ is the variance of the distribution  $\mu_T$. In this case, if we let
 $$\tilde g_\t(i,\rho)=i^2-(\sum_{j=1}^mj\rho(j))^2.$$
$\sum_{i=1}^m\tilde g_\t(i,\mu_T)\mu_T(i)$ is the variance of the distribution  $\mu_T$ as well. Thus  $g$ and $\tilde g$ give the same terminal functional in $\BV$. While in the process of deriving the fixed-point, the  terminal conditions in the Hamilton-Jacobi equation will be different. As a consequence, the time-inconsistent equilibrium will be different as well. Therefore it is natural to question which of those is the correct one to use. In fact, in  Section \ref{sec:mfg}, we introduce  a time-inconsistent mean-field game with infinite-many equivalent players.  We can clearly identify the correct forms of $f$ and $g$ from the problem itself. Thus, generally speaking, the forms of $f$ or $g$ depend on the 
%real 
model used.
 }
\end{example}

\subsection{Distribution-independent Equilibrium
%: Given 
with  a {\it   
Priori} $\nu_\dbT$}
In this subsection, we derive the process to find a (classical) time-inconsistent equilibrium if   $\nu_\dbT$ in the cost functional is given {\it a priori}.  This is the step (b) from Remark \ref{rem2step}. Essentially, we are finding a time-inconsistent equilibrium of which  the idea is taken from
\cite{Yong2012b}. The main effort is to find a time-inconsistent HJ equation  derived from an $N$-player game. While in our paper, we present the HJ equation and verify the required local-optimality directly. Further details regarding the derivation of 
%If the readers are interested in the derivation 
of the time-inconsistent HJ equation can be found in \cite{Yong2012b}.  We first  present some assumptions.

Given $u\in\cU$, define an operator $\cQ^u_t:\cM\mapsto \cM$ by
$$\cQ^u_t[h](i):=\sum_{j=1}^mq^{u(i)}_t(i,j)h(j).$$

\noindent{\bf Assumption (B)}

\begin{itemize}
 \item[(B1)] There exists  $\bar f_{\t,t}:M\times\cP\mapsto\dbR^+$ and $\Psi_t:M\times U\mapsto\dbR^+$ such that
		$$f_{\t,t}(i,v;\rho)=\bar f_{\t,t}(i;\rho)+\Psi_t(i,v),$$
and that		$\Psi_t(i,\cdot)$ is  continuous on $U_t(i)$ with
		$$\sup_{i\in M}\sup_{0\leq t\leq T}\sup_{v\in U_t(i)}\Psi_t(i,v)<K_1.$$
	
	\item[(B2)] There exist constants $K_2, K_3\geq0$ such that $$\left\{\ba{ll}\ad 0\leq \bar f_{\t,t}(i,\rho),~g_\t(i,\rho)\leq K_2,\\ [2mm]
	\ns\ad  |\bar f_{\t,t}(i,\rho)-\bar f_{\t,t}(i,\rho')|+|g_\t(i,\rho)-g_\t(i,\rho')|\leq K_3 d(\rho,\rho').\ea\right.$$
		
\item[(B3)]
There exists a map $\psi_t:  \cM\mapsto \cU$ such that for any $h\in\cM$,
$$\psi_t[h](i):=\argmin_{v\in U_t(i)}\[\Psi_{t}(i,v)+\sum_{j=1  }^mh(j)q_t^v(i,j)\].$$ Moreover, For any $h,h'\in\cM$,
\bel{Lippsi}d_\cU(\psi_t[h],\psi_t[h']) \leq \kappa_2\Vert h-h'\Vert_\cM
\q\text{ and }\q\lim_{\e\rightarrow 0}|\psi_{t+\e}[h](i),\psi_t[h](i)|_U=0.\eel
\end{itemize}

\begin{remark}
	{\rm The definition of $\psi_t$ induces a map $\psi:C([0,T],\cM)\mapsto L^1([0,T],\cU) $ by point-wise by defining
$$(\psi[\theta])_t(i)=\psi_t[\theta_t](i).$$
Moreover, if $\theta\in  D([0,T],\cM)$, by Assumption (B), $\psi[\theta]\in D([0,T],\cU)$.
}
\end{remark}

\begin{example}
	{\rm  Here we present an example where Assumption (B) is true.  Let $$q_t^{v}(i,j)=\alpha_t(i,j)+\beta_t(j) v$$
and the action set $U=[-1,1]$. Suppose that $\Psi_t(i,v)= \frac12 v^2$, $\alpha_t(i,j)\geq 0$ if $i\neq j$ and
$$\sum_{j=1}^m\alpha_t(i,j)=\sum_{j=1}^m\beta_t(j)=0.$$
One can easily see that $U_t(i)\neq\emptyset$ since $0\in U_t(i)$. If $v, v'\in U_t(i)$, i.e.,
$$\alpha_{t}(i,j)+\beta_t(j) v,~\alpha_{t}(i,j)+\beta_t(j) v'\geq 0\text{ for }i\neq j,$$
then for any $\lambda\in[0,1]$
$$\alpha_t(i,j)+\beta_t(j)[\lambda v+(1-\lambda v')]\geq 0\text{ for } i\neq j.$$
This proves that $$\lambda v+(1-\lambda )v'\in U_t(i),\text{ for any }\lambda \in[0,1],$$
i.e., $U_t(i)$ is a convex subset. It  can also  be seen that $U_t(i)$ is closed. Thus $U_t(i)$ is a closed subinterval of $[-1,1]$, and as a result,
$$\psi_t(i)=\argmin_{v\in U_t(i)}\[\frac{v^2} 2+v\sum_{j=1}^mh(j)\beta_t(j)\]$$
is well-defined by the strong convexity. Moreover, if $\beta_t(i)$ is right-continuous with respect to $t$, \eqref{Lippsi} holds directly. }
\end{example}

Now we are ready to introduce the time-inconsistent HJ equation for our problem.
Given $\nu_\dbT\in C([0,T],\cP)$, consider the following time-inconsistent Hamilton-Jacobi equation,
\bel{HJB}\left\{\ba{ll}\ad\partial_t\Theta_{\t,t}(i)+f_{\t,t}(i,\pi_t(i);\nu_t)+\cQ_{t}^{\pi_t}[\Theta_{\t,t}](i)=0\\
\ns\ad \pi_t(i)=\psi_t[\Theta_{t,t}](i),~\Theta_{\t,T}(i)=g_{\t}(i;\nu_T)\ea \text{ for any }i\in M. \right.\eel

Define $\BA_{t_0,t_1}^\pi:\cM\mapsto\cM$ by
$$(\BA_{t_0,t_1}^\pi h)(i)=\sum_{j=1}^mh(j)p_{t_0,t_1}^\pi(i,j)$$
where $p_{t_0,t_1}^\pi(j)$ is defined in \eqref{BQtransition}. It is easy to see that \eqref{HJB} is equivalent to
\bel{HJB-2}\left\{\ba{ll}\ad\Theta_{\t,t}=\int_t^T\BA^\pi_{t,s}f_{\t,s}(\cdot,\pi_s(\cdot);\nu_s)ds+\BA^\pi_{t,T}g_\t(\cdot;\nu_T)\\
\ns\ad \pi_t(i)=\psi_t[\Theta_{t,t}](i),~\Theta_{\t,t}(i)=g_\t(i).\ea\right.\eel
\ms

In view of the form of the HJ equation, formally we can conclude that the solution   $\Theta\in \sM[0,T]^2$ where
$$\ba{ll}\ad \sM([0,T]^2:=\{\Theta:(\t,t,i)\in\dbT^2\times M\mapsto\Theta_{\t,t}(i)\in\dbR\big|~\Theta_{\t,\cdot}\in C([0,T],\cM) \text{ for each }\t\in\dbT\}.\ea$$
Note that it is not required  $\Theta_{t,t}(i)$  being continuous in $t$, or even measurable. Let us 
look at the following example.

\begin{example}{\rm Let $\dbT_1$ be a non-measurable subset of $[0,T]$. Let
	$$\Theta_{\t,t}=\BI(\t\in \dbT_1).$$
	Obviously  $ \Theta_{\t,t}$ is continuous with respect $t$ for each fixed $\t$. While $\Theta_{t,t}=\BI(t\in \dbT_1)$ is not  a measurable function.}
\end{example}

Therefore, to guarantee the regularity of $\Theta_{t,t}$, we assume the following conditions hold.

\noindent{\bf Assumption (C)}
	$\bar f_{\t,t}(i), ~g_\t(i) $ are right continuous with left limit with respect to $\t$ for each $(t,i)\in \dbT\times M$.

With such an assumption, we have the following theorem.

\begin{theorem}\label{mth-1} Under Assumptions {\rm (A), (B), (C)},
given any $\nu_\dbT\in C([0,T],\cP)$,	there exists a solution pair $(\Theta,\pi)\in\sM[0,T]^2\times D([0,T],\cU)$ for \eqref{HJB} with $\Theta_{t,t}(i)$ is right-continuous with left-limit for each $i\in M$. As a consequence,
$$\pi_t(i)=\psi_t(i,\Theta_{t,t})\in D([0,T],\cU).$$
\end{theorem}

\begin{proof}
	The proof is based on the fixed-point theory. Since $\nu_\dbT$ is given  {\it a priori}, we omit $\nu_\dbT$ in the proof. Given two $\theta,\theta'\in  D([0,T],\cM)$, let $(\Theta,\pi),(\Theta',\pi')$ be the solutions of
$$\left\{\ba{ll}\ad\partial_t\Theta_{\t,t}(i)+f_{\t,t}(i,\pi_t(i))+\cQ_{t}^{\pi_t}[\Theta_{\t,t}](i)=0\\
\ns\ad \pi=\psi[\theta],~ \Theta_{\t,T}(i)=g_{\t}(i)\ea \text{ for any }i\in M, \right.$$
and
$$\left\{\ba{ll}\ad\partial_t\Theta'_{\t,t}(i)+f_{\t,t}(i,\pi'_t(i))+\cQ_{t}^{\pi'_t}[\Theta'_{\t,t}](i)=0\\
\ns\ad\pi'=\psi[\theta'],~ \Theta_{\t,T}(i)=g_{\t}(i)\ea \text{ for any }i\in M. \right.$$

By \eqref{Lippsi},
$$\ba{ll}\ds\int_t^Td_\cU(\pi_s(i),\pi'_s(i))ds\ad\leq \kappa_3\int_t^T\Vert\theta_s-\theta_s'\Vert_\cM ds.\ea$$
and we have the contraction inequality
\bel{Thetacontraction}\ba{ll}\ad\|\Theta_{\t,t}-\Theta'_{\t,t}\|_\cM\\
\ns\ad\q\leq \int_t^T\|\BA^\pi_{t,s}f_{\t,s}(\cdot,\pi_s(\cdot))-\BA^{\pi'}_{t,s}f_{\t,s}(\cdot,\pi'_s(\cdot))\|_\cM ds+
\|\BA^\pi_{t,T}g_\t-\BA^{\pi'}_{t,T}g_\t\|_\cM\\
\ns\ad\q\leq K\int_t^Td_\cU(\pi_s,\pi'_s) ds\leq \kappa_3K\int_t^T\Vert\theta_s-\theta_s'\Vert_\cM ds.\ea \eel

Now we show that if $\theta\in  D([0,T],\cM)$, $\Theta_{t,t}$ is also right continuous with left limit. Then we note that
$$\ba{ll}\ad\|\Theta_{t+\e,t}-\Theta_{t+\e,t+\e}\|_\cM\\
\ns\ad\q\leq \|\Theta_{t+\e,t+\e}-\Theta_{t,t+\e}\|_\cM
+\|\Theta_{t,t+\e}-\Theta_{t,t}\|_\cM\\
\ns\ad \q\leq \int_t^T\|\BA^\pi_{t,s}f_{t+\e,s}(\cdot,\pi_s(\cdot))-\BA^{\pi}_{t,s}f_{t,s}(\cdot,\pi_s(\cdot))\|_\cM ds+
\|\BA^\pi_{t,T}g_{t+\e}-\BA^{\pi}_{t,T}g_{t}\|_\cM\\
\ns\ad\qq+\|\Theta_{t,t+\e}-\Theta_{t,t}\|_\cM.\ea$$
Using Assumption (C) and letting $\e\rightarrow 0^+$,
$$\lim_{\e\rightarrow0^+}\Vert\Theta_{t+\e,t+\e}-\Theta_{t,t}\Vert_\cM=0.$$
This proves that $\Theta_{t,t}$ is right continuous  with respect to $t$. Similarly, $\Theta_{t,t}$ has a left limit with respect to $t$.

Now we are ready to prove the existence and uniqueness of the solution by adopting the fixed point theory. Given $\theta^{(1)}\in  D([0,T],\cM)$, let $\Theta^{(1)}$ be the solution of \eqref{HJB}. Let $\theta^{(2)}_t(i)=\Theta^{(1)}_t(i)$ for each $(t,i)$. By our claim, we know that $\theta^{(2)}\in D([0,T],\cM)$. Then let $\Theta^{(2)}$ be the solution of \eqref{HJB} using $\theta^{(2)}$. Recursively repeating such process, one can get a sequence of functions $\{(\theta^{(n)},\Theta^{(n)}\}$. By \eqref{Thetacontraction}, we have
$$\ba{ll}\ad\sup_{t\leq s\leq T}\sup_{0\leq \t\leq T}\Vert\Theta^{(n)}_{\t,s}-\Theta^{(n+1)}_{\t,s}\Vert_\cM\\
\ns\ad\q \leq \kappa_3 K(T-t)\int_t^T\Vert\theta^{(n)}_s-\theta_s^{(n+1)}\Vert_\cM ds\\
\ns\ad\q\leq  \kappa_3 K(T-t)\sup_{t\leq s\leq T}\sup_{0\leq \t\leq T}\Vert\Theta^{(n-1)}_{\t,s}-\Theta^{(n)}_{\t,s}\Vert_\cM.\ea$$
As a result, $T-t<\d$ for small $\d>0$, there exists a limit pair $(\theta,\Theta)$ with $\Theta\in\sM[T-\d,T]^2$  and    $$\int_{T-\d}^T|\theta_s(i)|ds<\infty.$$
 For the whole time interval $[0,T]$, one can divide $[0,T]$ into $[T,T-\d]$, $[T-\d,T-2\d], \ldots$. Recursively from $T$ to $0$, one can see that there exists a unique solution pair $(\theta,\Theta)\in  D([0,T],\cM)\times\sM[0,T]^2$.  Moreover we have $\pi=\psi[\theta]\in D([0,T],\cU)$.
	
\end{proof}

\begin{proposition}
	For any $\nu_\dbT\in C([0,T],\cP)$, there exists a uniform constant
 $\kappa_2$ (independent of $\nu_\dbT$) such that
	\bel{bounduniformTheta}0\leq \Theta_{\t,t}(i)\leq \kappa_2.\eel
\end{proposition}

\begin{proof}
By Assumption (B2) and the representation of $\Theta$ in \eqref{HJB-2}, 
%it is obvious that 
\eqref{bounduniformTheta} holds with some uniform constant $\k_2$ independent of the choice of $\nu_\dbT$.
\end{proof}

Next, we prove that the strategy from \eqref{HJB} verifies the local-optimality \eqref{localoptimal}.

\begin{theorem}\label{optimality} Under Assumptions {\rm (A), (B), (C)},
	given a tuple  $(\nu_\dbT,\Theta,\pi)$, where $(\Theta,\pi)$ solves \eqref{HJB} with given priori $\nu_\dbT$, it follows that
	$$\ba{ll}\ad\liminf_{\e\rightarrow 0^+}\frac{V_t(i,\pi^\e\oplus\phi_{t+\e}[\pi];\nu_\dbT)-V_t(i,\phi_t[\pi];\nu_\dbT)}\e\geq0\\
	\ns\ad\qq\qq\qq\qq\text{ for any }(t,i,\pi^\e)\in \dbT\times M \times D([t,t+\e),\cU)\ea$$
	with $\pi^\e_t=u$ for any $\e>0$.
\end{theorem}

\begin{proof}
	Note that
	$$\ba{ll}\ad J_{t,t}(i,\pi^\e\oplus\phi_{t+\e}[\pi],\nu_\dbT)\\
	\ns\ad\q=\int_t^{t+\e}\sum_{j=1}^mf_{t,s}(j,\pi_s^\e(i);\nu_s)\mu_s^{t,i,\pi^\e}(j)ds+\sum_{j=1}^mJ_{t,t+\e}(j,\phi_{t+\e}[\pi];\nu_{\dbT})\mu_{t+\e}^{t,i,\pi^\e}(j).\ea$$
	Therefore
	$$\ba{ll}\ad J_{t,t}(i,\pi^\e\oplus\phi_{t+\e}[\pi],\nu_\dbT)-J_{t,t}(i,\phi_{t}[\pi],\nu_\dbT)\\
	\ns\ad\q=\int_t^{t+\e}\sum_{j=1}^m\(f_{t,s}(j,\pi^\e_s(i);\nu_s)\mu_s^{t,i,\pi_\e}(j)-f_{t,s}(j,\pi_s(j);\nu_s)\mu_s^{t,i,\pi}(j)\)ds\\
	\ns\ad\qq+\sum_{j=1}^mJ_{t,t+\e}(j,\phi_{t+\e}[\pi];\nu_{\dbT})\(\mu_{t+\e}^{t,i,\pi^\e}(j)-\mu_{t+\e}^{t,i,\pi}(j)\)\ea$$
	
	Since $\pi^\e$ and $\pi$ are right-continuous, it follows that \bel{local-1}\ba{ll}\ad\limsup_{\e\rightarrow0}\frac1\e\[\int_t^{t+\e}\sum_{j=1}^m\(f_{t,s}(j,\pi_s^\e(j);\nu_s)\mu_s^{t,i,\pi^\e}(j)-f_{t,s}(j,\pi(s);\nu_s)\mu_s^{t,i,\pi}(j)\)ds\]\\
	\ns\ad \q=f_{t,t}(i,u(i);\nu_t)-f_{t,t}(i,\pi_t(i);\nu_t).\ea\eel
	Note that $\Theta_{\t,t}(i)=J_{\t,t}(i,\phi_t[\pi];\nu_\dbT)$, we have
	$$\ba{ll}\ad\sum_{j=1}^mJ_{t,t+\e}(j,\phi_{t+\e}[\pi];\nu_{\dbT})\mu_{t+\e}^{t,i,\pi^\e}(j)-J_{t,t+\e}(i,\phi_{t+\e}[\pi];\nu_{\dbT})\\
	\ns\ad=\sum_{j=1}^m\Theta_{t,t+\e}(j)[\d_i\BP   _{t,t+\e}^{\pi^\e}](j)-\Theta_{t,t}(i)-\(\Theta_{t,t+\e}(i)-\Theta_{t,t}(i)\)\\
	\ns\ad =\sum_{j=1}^m\Theta_{t,t}(j)[\d_i\BP_{t,t+\e}^{\pi^\e}](j)-\Theta_{t,t}(i)+\sum_{j=1}^m(\Theta_{t,t}(j)-\Theta_{t,t}(j))\([\d_i\BP_{t,t+\e}^{\pi^\e}](j)-\d_i(j)\).\ea$$
Since $\pi^\e,\pi$ are right-continuous and by Assumption (A),
	\bel{local-2}\left\{\ba{ll}\ad\lim_{\e\rightarrow0}\frac1\e\(\sum_{j=1}^mJ_{t,t+\e}(j,\phi_{t+\e}[\pi];\nu_{\dbT})\mu_{t+\e}^{t,i,\pi^\e}(j)-\Theta_{t,t}(i)\)=\cQ^{u}_{t}[\Theta_{t,t}](i)\\ [2mm]
	\ns\ad \lim_{\e\rightarrow0}\frac1\e\(\sum_{j=1}^mJ_{t,t+\e}(j,\phi_{t+\e}[\pi];\nu_{\dbT})\mu_{t+\e}^{t,i,\pi}(j)-\Theta_{t,t}(i)\)=\cQ^{\pi_t}_{t}[\Theta_{t,t}](i).\ea\right.\eel

As a result,  by \eqref{local-1}, \eqref{local-2} and the definition of $\pi_t(i)=\psi_t[\Theta_{t,t}](i)$,
$$\ba{ll}\ad\liminf_{\e\rightarrow 0}\frac{J_{t,t}(i,\pi^\e\oplus\phi_{t+\e}[\pi],\nu_\dbT)-J_{t,t}(i,\phi_t[\pi],\nu_\dbT)}\e\\
\ns\ad\q= f_{t,t}(i,u(i);\nu_t)-f_{t,t}(i,\pi_t(i);\nu_t)+\cQ^u_{t}[\Theta_{t,t}](i)-\cQ^{\pi_t}_{t}[\Theta_{t,t}](i)\\
\ns\ad\q\geq0.\ea$$
The proof is complete.
	
\end{proof}

\subsection{Time-inconsistent Distribution-dependent Equilibrium}

In this section,  we  
%will raise 
use
the following two-step recursion to prove the existence and uniqueness of the time-inconsistent mean-field equilibrium.

Step-1: Given a $\nu_\dbT^{(1)}\in C([0,T],\cP)$ with initial $\nu^{(1)}_0=\rho$, let $(\Theta^{(1)},\pi^{(1)})$ be the solution pair of \eqref{HJB}.

Step-2: Using the strategy $\pi^{(1)}$, let $\nu^{(2)}_\dbT$ be the solution of the dynamic equation \eqref{dynamic} with initial $\nu^{(2)}_0=\rho$.

By recursively repeating Step-1 and Step-2, we can get a sequence of %tuples 
triples
$(\nu_\dbT^{(n)},\Theta^{(n)},\pi^{(n)})$. Now we aim to prove such a sequence is convergent in an appropriate sense. We need the following  lemma.

\begin{lemma}
	Under Assumptions {\rm (A), (B), (C)}, given $\nu_\dbT,\tilde \nu_\dbT\in C([0,T],\cP)$, let $(\Theta,\pi)$ and $(\tilde \Theta,\tilde\pi)$ be the solutions of \eqref{HJB} respectively. Then there exists a constant $\kappa_3>0$ such that
	\bel{cauchytheta}\sup_{0\leq t\leq T}\sup_{0\leq \t\leq T}\Vert\Theta_{\t,t}-\tilde\Theta_{\t,t}\Vert_\cM \leq \k_3\sup_{0\leq t\leq T}d(\nu_t,\tilde\nu_t).\eel
\end{lemma}

\begin{proof}
	Note that
	$$\ba{ll}\ad\|\Theta_{\t,t}-\tilde \Theta_{\t,t}\|_\cM\\
	\ns\ad\q\leq \int_t^T\Vert\BA^\pi_{t,s}f_{\t,s}(\cdot,\pi_s(\cdot);\nu_s)-\BA^{\tilde\pi}_{t,s}f_{\t,s}(\cdot,\tilde\pi_s(\cdot);\tilde\nu_s)\|_\cM ds+\|\BA^\pi_{t,T}g_\t(\cdot;\nu_T)-\BA^{\tilde \pi}_{t,T}g_\t(\cdot;\tilde\nu_T)\|_\cM\\
	\ns\ad\q\leq \int_t^T\Vert\BA^\pi_{t,s}f_{\t,s}(\cdot,\pi_s(\cdot);\nu_s)-\BA^{\tilde\pi}_{t,s}f_{\t,s}(\cdot,\pi_s(\cdot);\tilde\nu_s)\|_\cM ds\\
	\ns\ad\q\q+\int_t^T\Vert\BA^{\tilde \pi}_{t,s}f_{\t,s}(\cdot,\pi_s(\cdot);\nu_s)-\BA^{\tilde\pi}_{t,s}f_{\t,s}(\cdot,\tilde\pi_s(\cdot);\tilde\nu_s)\|_\cM ds\\
	\ns\ad\q\q+\|\BA^\pi_{t,T}g_\t(\cdot;\nu_T)-\BA^{\tilde \pi}_{t,T}g_\t(\cdot;\nu_T)\|_\cM+\|\BA^{\tilde \pi}_{t,T}g_\t(\cdot;\nu_T)-\BA^{\tilde \pi}_{t,T}g_\t(\cdot;\tilde\nu_T)\|_\cM\\
	\ns\ad\leq K(T+1)\sup_{t\leq s\leq T}d(\nu_s,\tilde\nu_s) +K\int_t^T\|\Theta_{s,s}-\tilde\Theta_{s,s}\|_\cM ds.\ea$$
	Using Grownwall's inequality, one can see that there exists a $\kappa_3 >0$
	$$\sup_{0\leq \t\leq T}\sup_{0\leq t\leq T}\Vert\Theta_{\t,t}-\tilde\Theta_{\t,t}\Vert_\cM \leq \k_3\sup_{0\leq t\leq T}d(\nu_t,\tilde\nu_t).$$
\end{proof}

Now we are ready to present  the main theorem of this paper.

\begin{theorem}  Under Assumption {\rm (A), (B), (C)}, if $\kappa_1\kappa_2\kappa_3T<1$, there exists a unique time-inconsistent distribution-dependent equilibrium $(\rho,\pi^*)\in\cP\times D([0,T],\cM)$.
	
\end{theorem}

\begin{proof} By \eqref{cauchytheta}, we have
	$$\ds\sup_{0\leq s\leq T}\sup_{0\leq \t\leq T}\Vert\Theta^{(3)}_{\t,s}-\Theta^{(2)}_{\t,s}\Vert_\cM\leq \kappa_3\sup_{0\leq s\leq T}d(\nu^{(3)}_s,\nu^{(2)}_s).$$
	By \eqref{estimu} and Assumption (B), we have
	$$\ba{ll}\ds\sup_{0\leq s\leq T}d(\nu^{(3)}_s,\nu^{(2)}_s)ds\ad\leq \kappa_1\int_0^Td_\cU(\pi^{(2)}_s,\pi^{(1)}_s) ds\\
	\ns\ad\leq \kappa_1\kappa_2\int_0^T\Vert\Theta_{t,t}^{(2)}-\Theta_{t,t}^{(1)}\Vert_\cM ds\\
	\ns\ad\leq \kappa_1\kappa_2T\sup_{0\leq s\leq T}\sup_{0\leq \t\leq T}\Vert\Theta^{(2)}_{\t,s}-\Theta^{(1)}_{\t,s}\Vert_\cM.\ea$$
	Thus we have
		$$\ds\sup_{0\leq s\leq T}\sup_{0\leq \t\leq T}\Vert\Theta^{(3)}_{\t,s}-\Theta^{(2)}_{\t,s}\Vert_\cM\leq \kappa_1\kappa_2\k_3T\sup_{0\leq s\leq T}\sup_{0\leq \t\leq T}\Vert\Theta^{(2)}_{\t,s}-\Theta^{(1)}_{\t,s}\Vert_\cM.$$
		If $\kappa_1\kappa_2\k_3T<1$, we can easily see that $\Theta_{\t,t}^{(n)}$ is a Cauchy sequence with limit $\Theta^*_{\t,t}$ and $\nu^{(n)}_\dbT$ is a Cauchy sequence in $C([0,T],\cP)$ with limit $\nu_\dbT$. At the same time, we can get a pair $(\rho,\pi^*)$ where
		$$\theta^*_t(i):=\Theta^*_{t,t}(i)\text{ and }\pi^*=\psi[\theta^*]\in D([0,T],\cM).$$
		By Lemma \ref{optimality}, $(\rho,\pi^*)$ is  the time-inconsistent distribution-dependent equilibrium.
\end{proof}

\begin{remark}
	{\rm We have established the time-inconsistent mean-field equilibrium for general cases. If the functions $J$ and $V$ in the cost function $\BJ$ in \eqref{cost} are independent of the distribution, such problem reduces to the classical time-inconsistent control problem for Markov chains. Then by Theorem \ref{mth-1}, one can find the   time-inconsistent equilibrium $\pi^*$  which is independent of the choice $\rho$. Moreover, the assumption on the constants $\k_1,\k_2,\k_3$ and $T$ are no longer needed.
%not required any more.  
}
\end{remark}

\section{Time-inconsistent Infinite-player  Mean-field Game}\label{sec:mfg}
In this section, we will introduce an infinite-player game of which the mean-field equilibrium is equivalent to the time-inconsistent distribution-dependent equilibrium defined in Definition \ref{timfe}.

\ms

On  a complete probability measure space $\{\Omega, \BP,\cF\}$, consider a finite-state controlled Markov with state space $M$ and transition probability
$$\BP(X_{t+\D t}=j|X_t=i;v)= \left\{\ba{ll}\ad q_t^v(i,j)\D t+o(\D t),\\
\ns\ad 1+q^v_t(i,i)\D t+o(\D t)\ea\right. $$
where  $v$ is taken in some action space $U$. Denote $\dbN=\{1,\dots, N\}$ and $\dbT=[0,T]$.

Let $\{X^n_\dbT:n\in\dbN\}$ be the dynamics of  $N$-independent players with the same transition as  $\{X_\dbT\}$ and differential initials. Let  the empirical measure process $\rho^{\dbN,-k}_\dbT$ be
$$\rho^{\dbN,-k}_t(i;\omega)=\frac1{N-1}\sum_{n\neq k}\BI[X^n_t(\omega)=i].$$
Let the set of admissible strategies $ L^1([0,T],\cU) $ be similarly defined as before and $D^\dbN([0,T],\cU)$  be the set of all admissible strategies for $N$ players.

A classical mean-field game problem is to find a Nash equilibrium $\pi^\dbN=\pi^1\oplus\cdots\oplus\pi^N\in \cU^\dbN[0,T]$ such that
\bel{Nash} \hat J^k_{t}(i,\phi_t[\pi^k];\rho_\dbT^{\dbN,-k})\leq \hat J^k_{t}(i,\phi_t[\pi^\e];\rho_\dbT^{\dbN,-k})\text{ for any } \pi^\e\in L^1([0,T],\cU)  \eel
 and $\hat J^k_t$ is the cost functional for player $k$ defined as
\bel{Nash-2} \hat J^k_{t}(i,\phi_t[\pi^k];\rho_\dbT^{\dbN,-k})=\dbE\[\int_t^T \hat f^k_{s}(X^k_s,\pi^k_s(X^k_s);\rho^{\dbN,-k}_s)ds+\hat g^k(X^k_T;\rho^{\dbN,-k}_T)\big|X^k_t=i,\pi^\dbN\]\eel
for some appropriate $\hat f^k$ and $\hat g^k$.

Different from \eqref{Nash-2},
we suppose the cost function for each player is in the same form corresponding with their dynamics respectively
  with an additional non-exponential discounting factor $\t$ in the form of
$$J^k_{\t,t}(i,\phi_t[\pi^k];\rho_\dbT^{\dbN,-k})=\dbE\[\int_t^T f_{\t,s}(X^k_s,\pi^k_s(X^k_s);\rho^{\dbN,-k}_s)ds+g_\t(X^k_T;\rho^{\dbN,-k}_T)\big|X^k_t=i,\pi\].$$
for some appropriate $ f$ and $ g$. Here every player is assumed to make their decisions according to the same time-inconsistent cost functional. Because of the time-inconsistency, it is impossible   to find the equilibrium strategy which verifies \eqref{Nash}. Thus for such a game, we aim to find a time-inconsistent $N$-player strategy in the following sense.

\begin{definition}\label{nashequi}{\rm  An $N$-player  strategy $\pi^\dbN=\pi^1\oplus\cdots\oplus\pi^N\in D^\dbN([0,T],\cU)$ is called a {\it time-inconsistent equilibrium} if
	$$\ba{ll}\ad\liminf_{\e\rightarrow0^+}\frac{ J^k_{t,t}(i,\pi^\e\oplus\phi_{t+\e}[\pi^k];\rho_\dbT^{\dbN,-k})-J^k_{t,t}(i,\phi_t[\pi^k];\rho_\dbT^{\dbN,-k})}{\e}\geq 0\\
	\ns\ad\qq\qq\qq\qq\qq\text{ for any }k\in\dbN\text{ and } \pi^\e\in D([t,t+\e),\cU).\ea$$
}
\end{definition}

The $N$-player equilibrium can be understood in the following way. Consider a player $\ell$,
if it is assumed
%supposed
that the strategies of the other players are known from the equilibrium, the  strategy  of player $\ell$
%you
finds in the time-inconsistent control problem, coincides with  player $\ell$'s
%your
strategy  determined in the equilibrium.
The above equilibrium essentially indicates that with all the strategies of the other players fixed, the $\ell$-player's strategy is locally optimal.

Let $N\rightarrow\infty$, i.e., the number of the players tends to infinity, and suppose that  $\rho_0^{\dbN,-k}\rightarrow \gamma$ for any $k$. It can be seen that all the players are equivalent in such game. Thus we can conclude that every player should obey the same strategy 
%which is written as 
$\pi$. By the law of large numbers, $\rho_\dbT^{\dbN,-k}$ converges to the curve  $\mu_\dbT^{\gamma,\pi}$  determined by
$$\frac{d\mu_t}{dt}=\mu_tQ_t^{\pi_t} \text{ with }\mu_0=\gamma.$$

In this case the equilibrium is fully determined by $(\gamma,\pi)$ since every player is equivalent in such time-inconsistent mean-field game. It is not difficult to see that $(\gamma,\pi)$ coincides with that in Definition \ref{timfe}. In this sense, 
%This is also why 
we can call the equilibrium defined in Definition \ref{timfe} a time-inconsistent mean-field equilibrium too.

\section{Concluding Remarks}\label{sec:cr}

We have explored the time-inconsistent distribution-dependent control problem for  continuous-time controlled Markov chains. Due to the time-inconsistency and distribution-dependence of the cost functional, the problems have some unique interesting features. We show that the equilibrium we find is essentially equivalent to the equilibrium of an infinite-player game with time-inconsistent distribution-dependent cost. Moreover, if the cost function is in the classical form, the time-inconsistent mean-field equilibrium coincides with the time-inconsistent equilibrium for distribution-independent model.

%We still can see that the theory is  in its infancy and 
It appears to be possible to improve the results from several angles. For example, can we consider similar problem if  the state space is countable?  Can time delay be incorporated in the setup.
All of these deserve further investigations.
%We  might deal with it in the other paper.

	\end{document}